\documentclass[12pt,twoside]{article}
\usepackage{amsmath,amssymb,mathrsfs,graphicx,bbm,amsfonts,amsthm,mathtools}
\mathtoolsset{showonlyrefs}
\usepackage[utf8x]{inputenc}
\usepackage{nicefrac}
\usepackage{xcolor}
\usepackage[normalem]{ulem}
\usepackage[english]{babel}
\usepackage{dsfont}

\usepackage{graphicx,subcaption}
\usepackage{url}
\usepackage{blindtext}

\usepackage[colorlinks=true, citecolor = green, linkcolor=blue, breaklinks, pdfauthor ={"Elena Magnanini}]{hyperref}
\usepackage{todonotes}
\usepackage{comment}

\usepackage{graphicx}
\usepackage[title]{appendix}%

\newtheorem{thm}{Theorem}[section]

\newtheorem{rem}[thm]{Remark}
\newtheorem{prop}[thm]{Proposition}
\newtheorem{cor}[thm]{Corollary}
\newtheorem{example}[thm]{Example}
\newtheorem{assumption}[thm]{Assumption}
\newtheorem{defn}[thm]{Definition}

\DeclareMathAlphabet{\mathbfit}{OML}{cmm}{b}{it}

\makeatletter
\@addtoreset{equation}{section}
\makeatother

\usepackage{titling}
    \pretitle{\centering\LARGE\scshape}
        \posttitle{\vskip 0.75cm}

    \predate{\vskip 0.75 cm \centering\large}
        \postdate{\par}
\newcommand{\cadlag}{\ifmmode\mathcal D\else c\`adl\`ag \fi}

\newcommand{\ba}{\begin{array}}
	\newcommand{\ea}{\end{array}}
\newcommand{\bea}{\begin{eqnarray}}
	\newcommand{\eea}{\end{eqnarray}}
\newcommand{\be}{\begin{equation}}
	\newcommand{\ee}{\end{equation}}

\def \z {{\z}}

\def \z {{\zeta}}

\def \1{\mathbbm{1}} % funziona solo con il package bbm.

\definecolor{Sea}{RGB}{70, 170, 180}
\author{
Ivan Kryven\thanks{Utrecht University,
i.v.kryven@uu.nl}, Elena Magnanini \thanks{WIAS Berlin, magnanini@wias-berlin.de} 
}

\title{Spatial Coagulation Systems with Mercer Kernels: Replicator Dynamics and Gelation}

\date{\today}

\usepackage[top=2.5cm, bottom=2.5cm,left=2.5cm,right=2.5cm]{geometry}

\begin{document}
\maketitle
\abstract{ We study a spatial version of the Smoluchowski coagulation equation in which particles carry both an integer mass and a spatial location, and coagulate according to a mass–space product kernel. After coagulation, the resulting particle inherits the sum of the parent masses and a single spatial location.
We provide a recursive formula for the solution and show that, while the first local mass moments are conserved up to gelation time at each location, the second local mass moments evolve according to a closed system of differential equations. When the spatial part of the kernel is of Mercer type, we further show that this system reduces to a replicator equation under a time rescaling and normalisation.
We then use the fact that the replicator equation admits monotone functionals to quantify how spatial heterogeneity influences the rate of gelation and derive bounds on the gelation time. We also discuss the implications of these results for several important classes of Mercer kernels, including radial, diffusion, and translation‑invariant kernels.}
\noindent  \bigskip
\\
{\bf Keywords:}  Spatial Smoluchowski equation, gelation, replicator dynamics, Mercer kernels. 
\\\\
{\bf AMS Subject Classification 2010:} 60K35,  35Q70, 82C22.

\section{Introduction}
Coagulation models are prevalent across numerous disciplines, spanning from physical chemistry (to describe polymer formation), to astrophysics, where they simulate galaxy formation. 
Introduced by Smoluchowski in 1916 \cite{Smoluchowski16}, the Smoluchowski coagulation equation is a classical mean-field model describing the (irreversible) evolution of cluster densities under pairwise coagulation dynamics. 
In its discrete form, if $\lambda_t(i)$ denotes the density of particles of size $i\in \mathbb N$ at time $t\geq 0$, the equation reads
\begin{equation}\label{eq:classic_smolu}
\frac{d}{dt}\lambda_t(i)
=
\frac12\sum_{j=1}^{i-1}K(j,i-j)\lambda_t(j)\lambda_t(i-j)
-
\lambda_t(i)\sum_{j\geq1}K(i,j)\lambda_t(j),
\end{equation}
where $K(i,j)$ is the coagulation kernel describing the interaction rate between clusters of sizes $i$ and $j$.

The analysis of the probabilistic counterpart came only in the late 1960s and 1970s, by the introduction of an interacting particle system known as the Marcus--Lushnikov process \cite{Lushnikov78,Marcus68,aldous99}, whose hydrodynamic limit concentrates on trajectories of the Smoluchowski equation, sometimes including an additional term encoding the `sol-gel' interaction (the \emph{Flory equation}) \cite{fournier-laurencot-09,fournier-giet-04}. Beyond hydrodynamic limits, a natural question concerns the fluctuations around the deterministic mean--field behavior, for which we refer for instance to \cite{K,AmorimArellanoJara2026}.

From the analytical perspective, an extensive literature has been devoted to the study of existence and uniqueness of Eq. \eqref{eq:classic_smolu}, depending on the conditions imposed on the interaction kernel \cite{mcleod-62-1, white-80, norris-coag-99, fournier-laurencot-06}. 
A central question in the theory concerns whether solutions conserve mass globally or not, depending on the kernel. In the latter case, we say that the system exhibits \emph{gelation} \cite{Fournier25, jeon98, rezakhanlou2013, escobedo-et-al-mass-cons}, namely there is a loss of mass due to the emergence of macroscopic particles by some time $t \geq 0$.
%%%%%%%%%%

 In the 2000s, Norris introduced a general framework for existence and uniqueness in the \emph{cluster coagulation model} \cite{norris-cluster-coag}, capable of incorporating multiple particle properties such as velocity, morphology, and spatial location. It is useful to distinguish two broad categories of such properties.
Some properties are additive under coagulation. Typical examples include mass, charge, or surface area. These can be incorporated through multicomponent coagulation models, in which particles carry vector-valued attributes that are simply added when clusters merge \cite{ferreira2021,heydecker2024,hoogendijk2024,ferreira2025}.
Other properties are inherently non-additive and require different analytical techniques. Spatial location may be an example: when two particles occupying different sites coagulate, the location of the newly formed cluster cannot be obtained by adding the locations of the parent particles. Instead, an additional mechanism is needed to determine how spatial information is updated after coagulation.  
Although spatial effects were already implicit in Smoluchowski’s original motivation, a treatment of coagulation models with explicit spatial dependence has emerged only recently.
In this work we focus on coagulation models whose interaction rates depend not only on particle masses but also on their spatial location, a regime we refer to as the \emph{inhomogeneous} setting. We refer the reader to \cite{AnIyMa24, AnIyMa24_bis, AKLP23, AKLP26} for recent developments in this direction.
%%%%%%%%

Our goal is to understand how spatial inhomogeneities influence the evolution of moments and the onset of gelation in a spatial extension of the discrete Smoluchowski equation.
In this model, particles are characterised by both a size $i\in\mathbb N$ and a location $x$ belonging to a finite or countable set $\mathcal{S}$. The coagulation rate is determined by a mass--location kernel, and after coagulation the resulting particle has the sum of the parent masses while inheriting its spatial location from one of the parent particles. Under suitable moment-control conditions, we show that the first \emph{local} moment is conserved and we characterise the evolution in time of the second moment. Then, we show that  
the unique solution of the spatial Smoluchowski equation satisfies a recurrence formula similar to \cite[Theorem~2.1 (iii)]{GK}, and under mild assumptions on the spatial part of the kernel, we characterize the gelation time by an effective interaction coefficient measuring how the second moment is distributed across space. Such characterisation allows us to produce bounds on the gelation time, which varies depending on the choice of the space $\mathcal{S}$, the spatial component of the kernel $K$ and the initial distribution of particle masses. A side result of this analysis reveals an intricate connection between spatial coagulation and the replicator equation from evolutionary game  \cite{hofbauer1998}. 

The paper is organized in three main parts. Section~\ref{sub:modA} introduces the spatial equation together with the technical assumptions that we need. Section~\ref{sec:results} contains the main results and a subsection fully dedicated to important examples, and Section~\ref{sec:proofs} is devoted to proofs.

\paragraph{Technical contribution.}
The first main contribution of the paper is to show that, for Mercer spatial kernels, the evolution equation of the normalized second-moment profile has a replicator dynamics structure (see Theorem~\ref{thm:replicator-mercer}, Eq.~\ref{eq:replicator-mercer}). 
The replicator equation \cite{hofbauer1998} is a classical objects in evolutionary game theory and population dynamics, where it describes the evolution of population frequencies under selection. In our setting, however, the replicator form is not imposed but emerges directly from the coagulation equation after normalizing local second moments by the global second moment and rescaling time.

The replicator formulation plays a central role in the analysis. In the Mercer-kernel setting, it implies that the effective interaction coefficient $\alpha(\tau)$, which governs the growth of the global second moment, is nondecreasing along the dynamics. This monotonicity leads to better bounds on the gelation time, sharper than those available for general spatial kernels \cite[Theorem~2.2 (ii)]{norris-cluster-coag}, and forms the second main contribution of the paper. The proofs combine moment identities, nonlinear ODE techniques, and variational arguments associated with positive semidefinite kernels.

\section{The spatial coagulation equation}\label{sub:modA}

Let $\mathcal S$ be a finite or countably infinite set of spatial sites.  
For particle sizes $i\in\mathbb N$ and sites $x\in\mathcal S$, we consider time-dependent cluster densities $\lambda_t(i,x)\ge 0$ evolving according to the spatial Smoluchowski equation\footnote{This equation coincides with \cite[Eq.~(7.22)]{AKLP23}, combined with the (second) choice of the placement kernel described in \cite[Remark~1.1]{AKLP23}.}
\begin{equation}\label{SE1}
\begin{aligned}
\frac{d}{dt}\lambda_t(i,x)
&=
\sum_{y\in\mathcal S}\sum_{j=1}^{i-1}
\frac{j}{i}\,
K\bigl((j,x),(i-j,y)\bigr)\,
\lambda_t(j,x)\lambda_t(i-j,y)\\
&\hspace{3cm}-
\lambda_t(i,x)
\sum_{y\in\mathcal S}\sum_{j\ge1}
K\bigl((i,x),(j,y)\bigr)\lambda_t(j,y),
\end{aligned}
\end{equation}
with a normalised monodisperse initial condition
\begin{equation}\label{IC}
\sum_{i\ge1} i\,\lambda_0(i,x)=\lambda_0(1,x),
\qquad
\sum_{x\in\mathcal S}\lambda_0(1,x)=1.
\end{equation}

Compared with standard multitype coagulation models, where clusters consist of singletons of different types, here each cluster occupies a single site. When clusters $(j,x)$ and $(i-j,y)$ coagulate, the resulting cluster has size $i$, and its location is chosen from $\{x,y\}$ proportionally to the sizes of the merging clusters:
\[
\mathbb P[\text{new site}=x]=\frac{j}{i}.
\]
We define the local and global size moments
\begin{equation}\label{moments}
m_k(x,t):=\sum_{i\ge1} i^k\lambda_t(i,x),
\qquad
M_k(t):=\sum_{x\in\mathcal S} m_k(x,t),
\qquad k\in\mathbb N.
\end{equation}
%%%%%%%%%%
It may happen that, after a finite time, the total mass $M_1$ is no longer conserved. This phenomenon, known as \emph{gelation}, is interpreted as the emergence of one or more giant clusters carrying a positive fraction of the total mass. A standard way to characterise such phase transition is through the behaviour of the second moment: gelation occurs when $M_2$ blows up in finite time. Accordingly, we introduce the following definition (see also \cite[Def.~2(b)]{jeon98}).

\begin{defn}[Gelation time]\label{def:tg}
The gelation time of a solution of the Smoluchowski equation \eqref{SE1} is defined by 
\begin{equation}\label{eq:def:tg}
t_{\mathrm{gel}} := \sup\{t>0:\ M_2(t)<\infty\}.
\end{equation}
\end{defn}
Now we assume a multiplicative kernel in the size variable,
\begin{equation}\label{kernel}
K\bigl((i,x),(j,y)\bigr)=ij\,G(x,y),
\qquad x,y\in\mathcal S,
\end{equation}
where the spatial kernel $G:\mathcal S\times\mathcal S\to(0,\infty)$ is
uniformly bounded. Under this choice, \eqref{SE1} becomes
\begin{equation}\label{SE_prod}
\frac{d}{dt}\lambda_t(i,x)=
\sum_{y\in\mathcal S}\sum_{j=1}^{i-1} \frac{j}{i}\,j(i-j)\,G(x,y) \lambda_t(j,x)\lambda_t(i-j,y)-i\,\lambda_t(i,x) \sum_{y\in\mathcal S}\sum_{j\ge1}
j\,G(x,y)\lambda_t(j,y).
\end{equation}
%%%%%%%%%%%%
Since the kernel is uniformly bounded and \emph{approximately multiplicative} in the sense of
Norris, and since $M_2(0)<\infty$, the spatial Smoluchowski equation
\eqref{SE_prod} admits a unique solution up to the gelation time
$t_{\mathrm{gel}}$, and the second moment diverges as $t\to
t_{\mathrm{gel}}$ \cite[Theorem~2.2]{norris-cluster-coag}.

\medskip

Throughout the paper we will assume, when needed, one or both of the following
moment–control conditions.

\begin{assumption}\label{ass:moment-justification}
There exists $T_1>0$ such that the following hold for all $t\in[0,T_1)$.

\smallskip
\noindent\hypertarget{ass:A1}{\textnormal{(A1)}}  
For every $x\in\mathcal S$, the series

\[
\sum_{y\in\mathcal S} G(x,y)\,m_2(x,t)\,m_1(y,t)
\]

converges uniformly in $t$.

\smallskip
\noindent\hypertarget{ass:A2}{\textnormal{(A2)}}  
For every $x\in\mathcal S$, the series

\[
\sum_{y\in\mathcal S} G(x,y)\,m_2(x,t)\,m_2(y,t)
\]

converges uniformly in $t$.
\end{assumption}

\begin{rem}
If

\[
\sup_{x\in\mathcal S}\sum_{y\in\mathcal S}G(x,y)<\infty
\quad\text{and}\quad
\sup_{t\in[0,T_1)} m_2(x,t)<\infty
\ \text{for all }x\in\mathcal S,
\]
then \hyperlink{ass:A1}{(A1)} holds automatically.  
In particular, if $\mathcal S$ is finite, \hyperlink{ass:A1}{(A1)} is always satisfied.
\end{rem}

In the non-spatial case, when $\mathcal S$ consists of a single site,
condition \hyperlink{ass:A1}{(A1)} reduces to uniform boundedness of $M_2(t)$ on $[0,T_1)$, the
classical setting studied for example in \cite[Theorem~1]{mcleod-62-1}.

\section{Results}\label{sec:results}
Our main goal is to understand how the underlying space $\mathcal S$ and the associated spatial kernel $G$ influence the behaviour of the coagulation equation~\eqref{SE_prod} under suitable summability assumptions.  
A first observation is that the local first moment is conserved: for each site $x\in\mathcal S$, the total mass $\sum_{i\ge1} i\,\lambda_t(i,x)$
remains equal to its initial value. Thus, the dynamics do not induce any mass transport across space.  
In contrast, the local and global second moments satisfy a coupled system of differential equations in which the kernel $G$ appears explicitly. In this sense, $G$ determines the strength and structure of the interaction between sites and plays a central role in the onset of gelation.
\begin{prop}[Moments]\label{prop:moment-equations}
Under assumption \hyperlink{ass:A1}{(A1)},  the first moment is conserved at each site for $t\in[0,T_1)$,
\begin{equation}\label{eq:consm1}
m_1(x,t)=\lambda_0(1,x),
\qquad x\in\mathcal S,
\end{equation}
and the second  local and global moments satisfy differential equations:
\begin{align}
\frac{d}{dt} m_2(x,t)&= m_2(x,t)\sum_{y\in\mathcal S} G(x,y)m_2(y,t), \label{eq:m2-local-prop} \\
\frac{d}{dt} M_2(t)&= \sum_{x,y\in\mathcal S} G(x,y)m_2(x,t)m_2(y,t).\label{eq:M2-global}
\end{align}
\end{prop}
The conservation of the first moment allows one to simplify equation \eqref{SE_prod}
and obtain an equivalent on $[0,T_1]$ triangular system of differential equations, each coupled to the preceding ones.  Exploiting this triangular structure, we write a recursive
representation of the solution.

\begin{thm}[Recursive solution]
\label{thm:spatial-multiplicative}
Assume \hyperlink{ass:A1}{(A1)}. Then for each $k\ge1$, $\lambda_t(k,x)$ satisfies a non-homogeneous linear differential equation,
\begin{equation}\label{eq:triangular-thm}
\frac{d}{dt}\lambda_t(k,x)+k\,p(x)\lambda_t(k,x)=g_k(x,t), 
\qquad t\in[0,T_1),
\end{equation}
where
\[
p(x):=\sum_{y\in\mathcal S}\lambda_0(1,y)\,G(x,y),
\]
and
\begin{equation}\label{eq:gk-thm}
g_k(x,t)=\frac1k
\sum_{y\in\mathcal S} G(x,y)
\sum_{j=1}^{k-1} j^2(k-j)\,
\lambda_t(j,x)\lambda_t(k-j,y).
\end{equation}
In particular, the solution is recursively given by
\begin{equation}\label{lambdak}
\lambda_t(k,x)
=
e^{-k t p(x)}
\left(
\lambda_0(k,x)
+
\int_0^t e^{k s p(x)}\, g_k(s,x)\,ds
\right),
\qquad k\ge1.
\end{equation}
\end{thm}

Proposition~\ref{prop:moment-equations} and Theorem~\ref{thm:spatial-multiplicative}
implicitly describe the evolution of $\lambda_t(i,x)$ and its moments but do not yet give either the value of the
gelation time or the mechanism for gelation onset.  A first observation is that the right-hand side of~\eqref{eq:M2-global} is strictly
positive for all $t\geq 0$.  Consequently, the global second moment $M_2(t)$ is monotonically increasing, and
gelation, if it occurs, corresponds to the blow-up of $M_2(t)$ in finite time. This observation becomes more insightful when we rewrite the differential equation for $M_2(t)$ in terms of normalised local second moments.
\begin{thm}[Gelation time]
\label{thm:riccati-general}
Assume Assumption~\ref{ass:moment-justification}.
For every $t$ such that $M_2(t)\in(0,\infty)$, define normalised local second moment
\begin{equation}\label{eq:pt}
p_t(x)=\frac{m_2(x,t)}{M_2(t)}.
\end{equation}
Then $M_2(t)$ satisfies 
\begin{equation}
\label{eq:Riccati-general}
\frac{d}{dt}M_2(t)=\alpha(t)M_2(t)^2
\end{equation}
where
\begin{equation}\label{eq:alfa}
\alpha(t)=\sum_{x,y\in\mathcal S} G(x,y)p_t(x)p_t(y).
\end{equation}
Therefore the gelation time is given by
$
t_{\mathrm{gel}}
=
\inf
\Big\{
t>0:\int_0^t\alpha(s)ds=M_2(0)^{-1}
\Big\}.
$
\end{thm}
Since $m_2(x,t)\ge 0$ for all $x\in\mathcal S$, the normalised quantity
$p_t(x)=m_2(x,t)/M_2(t)$ defines a probability measure on $\mathcal S$.
Theorem~\ref{thm:riccati-general} then shows that $M_2(t)$ satisfies a
Riccati-type differential equation whose rate depends on how the second
moment is distributed across the sites of $\mathcal S$ through the weights
$p_t(x)$.   Because $\mathcal S$ is countable, this rate can be bounded above and below
using the supremum and infimum of the spatial kernel $G$, which in turn
yields explicit upper and lower bounds on the gelation time.
\begin{cor}[Bounds on $t_{\rm gel}$]
\label{cor:sharp-bounds} Assume
Assumption~\ref{ass:moment-justification}. Then, for all $t$ such that
$M_2(t)\in(0,\infty)$,
\begin{equation}\label{eq:alpha_bound}
\inf_{x,y\in\mathcal S} G(x,y)\le \alpha(t)\le \sup_{x,y\in\mathcal S} G(x,y),
\end{equation}
and consequently,
\begin{equation}\label{bounds_Tgel}
\big(\sup_{x,y\in\mathcal S} G(x,y)M_2(0)\big)^{-1} \le t_{\mathrm{gel}} \le \big(\inf_{x,y\in\mathcal S} G(x,y)M_2(0)\big)^{-1},
\end{equation}
with the convention that the upper bound is $+\infty$ if
$\inf_{x,y\in\mathcal S} G(x,y)=0$. 
\end{cor}
In particular, it can be shown that if a non-constant spatial kernel is uniformly bounded by $1$, it can only postpone gelation comparing to the non-spatial case:
\begin{cor}%Bounded kernels do not accelerate gelation
\label{cor:normalized-kernel}
Assume the hypotheses of Theorem~\ref{thm:riccati-general}, and suppose that
$$
\sup_{x,y\in\mathcal S} G(x,y)\le 1.
$$
Then
$
t_{\mathrm{gel}}
\ge
\frac{1}{M_2(0)}.
$
That is, the gelation does not occur earlier than in the non-spatial case.
\end{cor}

On the one hand, the lower and upper bounds in~\eqref{bounds_Tgel} are sharp in
the sense that equality is attained simultaneously if and only if
$G(x,y)=C$ for all $x,y\in\mathcal S$, for some constant $C>0$.  
In this case $\alpha(t)\equiv C$ for all $t$, and the differential equation
\[
\frac{d}{dt}M_2(t)=C\,M_2(t)^2
\]
can be solved explicitly, yielding $t_{\mathrm{gel}}=1/(C\,M_2(0))$, so that
the two bounds in~\eqref{bounds_Tgel} coincide.  On the other hand, the bounds in~\eqref{bounds_Tgel} recover the classical
estimates obtained by Norris for approximately multiplicative kernels \cite[Thm.~2.2(ii)]{norris-cluster-coag},
restricted here to our spatial setting in which the coagulation rate is
bounded above and below by multiples of the product of masses.  
In our formulation, the quantities $\inf G$ and $\sup G$ play the role of
these constants.   The absence of further refinement is due to the fact that even though our coagulation kernel $K$ is specialised to be multiplicative in size,  its spatial part $G$ is considered in full generality.

By additionally requiring the spatial kernel $G$ to be symmetric and positive
semidefinite in an appropriate sense, one can analyse the evolution of the
probability measure $p_t(x)$ and show that the rate in the Riccati-type
equation~\eqref{eq:Riccati-general} is nondecreasing.  
This monotonicity property allows us to improve the bounds on the gelation
time by incorporating information about the initial spatial distribution of
the second moment, together with a more nuanced characteristic of the kernel
than its supremum.  
To this end, we further specialise the spatial kernel $G$ to admit a Mercer
representation.

\begin{defn}[Mercer spatial kernel]\label{def:mercer}
The spatial kernel $G$ is of Mercer type if there exist a
real Hilbert space $(\mathcal H,\langle\cdot,\cdot\rangle)$ and a feature map
$\Phi:\mathcal S\to\mathcal H$ such that
\begin{equation}\label{eq:mercer}
G(x,y)=\bigl\langle \Phi(x),\Phi(y)\bigr\rangle_{\mathcal H},
\qquad x,y\in\mathcal S.
\end{equation}
In particular, $G$ is then symmetric and positive semidefinite.
We use notation
\[
(Gp_t)(x):=\sum_{y\in\mathcal S} G(x,y)p_t(y).
\]
\end{defn}
\begin{rem}[Geometric interpretation]
The Mercer representation also yields a natural geometric interpretation of the effective interaction coefficient $\alpha(t)$ in terms of a barycenter in feature space. Indeed,
\[
\alpha(t)\overset{\eqref{eq:alfa}}{=}
\sum_{x,y\in\mathcal S}G(x,y)p_t(x)p_t(y)
=
\Bigl\|
\sum_{x\in\mathcal S}p_t(x)\Phi(x)
\Bigr\|_{\mathcal H}^2,
\]
so that \(\alpha(t)\) is the squared norm of the barycentre associated with the second-moment profile \(p_t\). Since by definition \eqref{eq:pt} 
\(p_t(x)\propto m_2(x,t)\),
this quantity becomes large at sites containing large clusters. Thus the growth of the second moment depends not only on the interaction strengths of the kernel \(G(x,y)\), but also on the spatial sites containing most of the large clusters. 
\end{rem}

\begin{thm}[Gelation Bound for Mercer kernels]\label{thm:replicator-mercer}
 Let $G$ be Mercer, in particular symmetric and bounded, and consider rescaled time 
 \begin{equation}\label{eq:tau-def}
\tau(t):=\int_0^t M_2(s) {\rm d}s.
\end{equation}
  Then, under Assumption~\ref{ass:moment-justification},  for every $t$ such that $M_2(t)\in(0,\infty)$, the measure
$p_\tau(x)$ satisfies the replicator equation 
\begin{equation}\label{eq:replicator-mercer}
\frac{d}{d\tau}p_\tau(x)
=
p_\tau(x)\Bigl((Gp_\tau)(x)-\alpha(\tau)\Bigr),
\qquad x\in\mathcal S,
\end{equation}
where $\tau\mapsto \alpha(\tau)$, as defined in Theorem \ref{thm:riccati-general}, is nondecreasing. Moreover,
the gelation time $t_{\mathrm{gel}}$ satisfies
\begin{equation}\label{eq:bounds-mercer}
\frac{1}{\alpha_{\max}M_2(0)}
\leq
t_{\mathrm{gel}}
\leq
\frac{1}{\alpha(0)M_2(0)},
\end{equation}
where
\begin{equation}\label{eq:alpha-max}
\alpha_{\max}
:=
\sup_{\substack{q:\mathcal S\to[0,1]\\ \sum_x q(x)=1}}
\sum_{x,y\in\mathcal S} G(x,y)q(x)q(y)<\infty,
\end{equation}
and we assume the convention that the upper bound is $+\infty$ if $\alpha(0)=0$.
\end{thm}

Thus the bounds in~\eqref{eq:bounds-mercer} strictly improve the crude kernel
bounds~\eqref{bounds_Tgel}, except in the degenerate case where
$\alpha(0)$ and $\alpha_{\max}$ coincide with the extremal values of $G$.
Indeed,
\[
\inf_{x,y}G(x,y)\le \alpha(0)\le\alpha_{\max}\le\sup_{x,y}G(x,y),
\]
and therefore
\[
\frac{1}{\alpha_{\max}M_2(0)}
\;\ge\;
\frac{1}{\sup_{x,y}G(x,y)\,M_2(0)},
\qquad
\frac{1}{\alpha(0)M_2(0)}
\;\le\;
\frac{1}{\inf_{x,y}G(x,y)\,M_2(0)}.
\]
It is also instructive to interpret the value of $\alpha_{\max}$.  
Let $X$ be a random variable taking values in $\mathcal S$ with
$\mathbb P(X=x)=q(x)$.  
Then $\mathbb E[\Phi(X)]$ is the \emph{mean feature vector}, and
\[
V(q):= \sum_{x,y\in\mathcal S} G(x,y)\,q(x)\,q(y)
	=\sum_{x,y\in\mathcal S}\langle \Phi(x),\Phi(y)\rangle_{\mathcal H}\,q(x)\,q(y)
         =\bigl\|\mathbb E[\Phi(X)]\bigr\|_{\mathcal H}^2.
\]
Consequently,
\[
\alpha_{\max}=\sup_{\substack{q:\mathcal S\to[0,1]\\ \sum_x q(x)=1}}V(q)=
\left(
\sup_{q}
\bigl\|\mathbb E[\Phi(X)]\bigr\|_{\mathcal H}
\right)^2,
\]
that is, $\alpha_{\max}$ is the largest squared norm of a mean feature vector
attainable by any probability distribution on~$\mathcal S$.

Finally, several observations about the evolution of the second-moment
distribution $p_t(x)$ follow from the replicator equation
\eqref{eq:replicator-mercer}.  
An initial distribution $p_0(x)$ is stationary when it is a fixed
point of the ODE~\eqref{eq:replicator-mercer}, in which case
$\alpha(t)=\alpha(0)$ for all $t\in[0,T_1]$.  
For such initial data, the gelation time coincides with the upper bound
in~\eqref{eq:bounds-mercer}.

\begin{cor}\label{cor:fixed-point-gelation}
Assume conditions of Theorem~\ref{thm:replicator-mercer} and suppose additionally that for all $x\in\mathcal S$, $p_0(x)$ is such that
$$p_0(x)(Gp_0(x)-\alpha(0))=0.$$ Then for all $t$ such that $M_2(t)<\infty$, $p_t(x)\equiv p_0(x)$, $\alpha(t)\equiv\alpha(0)$ and if $\alpha(0)>0$, the gelation time is given  by
\[
t_{\mathrm{gel}}=\frac{1}{\alpha(0)M_2(0)},
\]
while if $\alpha(0)=0$ there is no gelation (i.e.\ $M_2(t)$ remains finite for all $t$).

In particular, assume $\mathcal S$ is finite and $p_0(x)=\frac{1}{|\mathcal S|},$ for $x\in\mathcal S.$ Then $p_0$ is a fixed point of the replicator equation if and only if $G$ is row‑stochastic, that is
\[
\sum_{y\in\mathcal S}G(x,y)=\sum_{y\in\mathcal S}G(x',y) \qquad\text{for all }x,x'\in\mathcal S.
\]
\end{cor}

In the non-stationary regime, the replicator dynamics modifies the profile \(p_t(x)\) according to
Eq.~\eqref{eq:replicator-mercer}.
Now, recall from \eqref{eq:pt} that \(p_t(x)\) measures the fraction of the total second moment located at the site \(x\), and therefore describes where the large clusters are concentrated. The quantity \((Gp_t)(x)\) can then be interpreted as the average interaction strength between the site \(x\) and the sites currently carrying most of the large clusters. Sites with \((Gp_t)(x)>\alpha(t)\) gain mass in the profile \(p_t\), while sites with \((Gp_t)(x)<\alpha(t)\) lose mass. 
Consequently, the replicator dynamics shifts mass in $p_t(x)$
towards states that yield larger values of $\alpha(t)$, and the trajectory $p_\tau$ converges, as $\tau\to\infty$, to one of the stationary KKT (Karush-Kuhn-Tucker) points of $V(q)$ on the simplex $\sum_{x\in\mathcal S} q(x)=1$ \cite{hofbauer1998}. 
Depending on the support of these KKT points, gelation may either occur simultaneously in the whole $\mathcal S$ or localised at specific sites. 

Finally, recall that the rescaled time variable is defined by
$
\tau(t)=\int_0^t M_2(s)\,ds.$
Hence \(\tau(t)\to\infty\)  when \(M_2(t)\to\infty\). Therefore, even when the gelation time \(t_{\mathrm{gel}}\) is finite, the asymptotic regime \(\tau\to\infty\) corresponds, in the original time scale, to approaching the gelation time, \(t\to t_{\mathrm{gel}}\). 

\subsection{Examples}
In what follows we discuss several special cases of Mercer kernels, using the
notation introduced above throughout the examples. For convenience, we recall the quantities introduced above. Let
\[
\mu_t:=\sum_{x\in\mathcal S} p_t(x)\Phi(x).
\]
Then,
\[
\alpha(t):=\sum_{y,x\in\mathcal S} p_t(x)p_t(y)G(x,y)
=\langle\sum_{x\in\mathcal S} p_t(x)\Phi(x),\sum_{x\in\mathcal S} p_t(x)\Phi(x) \rangle = \|\mu_t\|_{\mathcal H}^2.
\]
\[
\alpha_{\max}:=\sup_{\substack{q:\mathcal S\to[0,1]\\ \sum_x q(x)=1}}
\sum_{x,y\in\mathcal S} G_{x,y}q(x)q(y)=\sup_{x\in\mathcal S}\|\Phi(x)\|_{\mathcal H}^2,
\]
where the last equality follows from the fact that $G$ is positive semidefinite. Indeed, the associated quadratic form is convex in $q$, and therefore its supremum over the simplex of probability measures is attained at an extreme point. To connect back to the classical non-spatial setting, now expressed in the
Mercer-kernel notation, it suffices to choose
$\Phi(x):=\sqrt{C}\in\mathbb R$ for every $x\in\mathcal S$.  
Then $\alpha(t)\equiv C$ and we recover the well-known formula
$t_{\mathrm{gel}}=1/(C\,M_2(0))$.  \\

 \begin{example}[Rank-one]\label{ex:rank-one-mercer}
Let $w:\mathcal S\to(0,\infty)$ be bounded and set $G(x,y)=w(x)w(y)$.
Then $\Phi(x)=w(x)$ and $\alpha(t)=\big(\sum_x w(x)p_t(x)\big)^2.$ So that
\[
\frac{1}{(\sup_x w(x))^2 M_2(0)}\le t_{\mathrm{gel}}\le\frac{1}{\big(\sum_x w(x)p_0(x)\big)^2\,M_2(0)}.
\]
If $w$ has a unique maximiser $x_0$, then $\alpha_{\max}=w(x_0)^2$ and
the lower bound is sharp. This setting is straightforward to generalise to a linear combination of several rank one kernels.
\end{example}
\begin{example}[Finite-dimensional feature map]\label{ex:finite-dim-mercer}
Let $\psi:\mathcal S\to\mathbb R^d$ be bounded and define
$G(x,y)=\psi(x)\cdot\psi(y)$. Then $\Phi(x)=\psi(x)$ and $\alpha(t)=\|\mu_t\|_{\mathbb R^d}^2, $ for $\mu_t=\sum_x p_t(x)\psi(x).$
 Suppose further that $\psi$ attains its maximal norm at some $x_{\max}$, then $\alpha_{\max}=\|\psi(x_{\max})\|^2$ and 
\[
\left(\|\psi(x_{\max})\|^2 M_2(0)\right)^{-1}\le t_{\mathrm{gel}}\le\Big(\|\sum_x p_0(x)\psi(x)\|^2M_2(0)\Big)^{-1}.
\]

\end{example}

\begin{example}[Translation-invariant kernel]
Let $\mathcal S=\mathbb Z/L\mathbb Z$ and let
$g:\mathcal S\to\mathbb R$ be even and positive semidefinite in the sense
that its (unitary) discrete Fourier transform
\[
\hat g(k)
=
\frac{1}{\sqrt{L}}\sum_{\Delta\in\mathcal S}
g(\Delta)e^{-2\pi i k\Delta/L}
\]
is nonnegative for all $k\in\mathcal S$.
Then $G(x,y):=g(x-y)$, is a Mercer kernel with feature map
\[
\Phi(x)=L^{-1/4}\bigl(\sqrt{\hat g(k)}\,e^{2\pi i kx/L}\bigr)_{k\in\mathcal S}.
\]
For the unitary Fourier transform of $p_t(x)$ we have
\[
\alpha(t)
=
\sum_{k\in\mathcal S}
\hat g(k)\,|\widehat{p_t}(k)|^2,
\qquad
\alpha_{\max}
=
\max_{k\in\mathcal S}\hat g(k),
\]
and hence
\[
\frac{1}{M_2(0)\max_{k\in\mathcal S}\hat g(k)}
\le
t_{\mathrm{gel}}
\le
\frac{1}{M_2(0)\sum_{k\in\mathcal S}
\hat g(k)|\widehat{p_0}(k)|^2}.
\]
\end{example}

\begin{example}[Diffusion kernel]\label{ex:semigroup-mercer}
Let $(P_s)_{s>0}$ be a reversible Markov semigroup on $\mathcal S$ with
reversible measure $\pi$, so that $\pi(x)P_s(x,y)=\pi(y)P_s(y,x),$ for $ x,y\in\mathcal S.$
Define the symmetrised diffusion kernel
\[
G(x,y):=\frac{P_s(x,y)}{\sqrt{\pi(x)\pi(y)}}.
\]
Then $G$ is a symmetric positive semidefinite kernel. Let $K$ be an index set
with $|K|=|\mathcal S|$. By the spectral theorem, there exists an orthonormal basis $(\psi_k)_{k\in K}$ of $\ell^2(\mathcal S)$ and eigenvalues $\lambda_k\in[0,1]$ such that
\[
G(x,y)=\sum_{k\in K} \lambda_k^s\,\psi_k(x)\psi_k(y),
\]
so a feature map is given by
\[
\Phi(x):=\bigl(\sqrt{\lambda_k^s}\,\psi_k(x)\bigr)_{k\in K}.
\]
For a probability distribution $p_t$ on $\mathcal S$, define the spectral
coefficients
\[
\widehat{p_t}(k):=\sum_{x\in\mathcal S} p_t(x)\,\psi_k(x).
\]
Then
\[
\alpha(t)
=
\sum_{k\in K} \lambda_k^s\,|\widehat{p_t}(k)|^2
\le
\bigl(\max_{k\in K}\lambda_k^s\bigr)
\sum_{k\in K} |\widehat{p_t}(k)|^2
\le
\max_{k\in K}\lambda_k^s,
\]
and since $\sum_{k\in K}|\widehat{p_t}(k)|^2=\sum_{x\in\mathcal S}p_t(x)^2\le1$,
by Theorem~\ref{thm:replicator-mercer}, the gelation time satisfies
\[
\bigl(M_2(0)\,\max_{k\in K}\lambda_k^s\bigr)^{-1}
\le t_{\mathrm{gel}}\le
\biggl(
M_2(0)\,\sum_{k\in K}\lambda_k^s\,|\widehat{p_0}(k)|^2
\biggr)^{-1}.
\]
In this sense, the slowest diffusion mode (largest eigenvalue $\lambda_k$)
controls a universal lower bound on the gelation time, while the upper bound
is smaller when the initial distribution $p_0$ is concentrated on rapidly
decaying modes.
\end{example}

Finally, the following example illustrates that sometimes it is enough to be able to compute the norm of $\Phi$ without explicitly constructing $\Phi$, as such a feature map exists.
\begin{example}[Radial kernel on a metric space]\label{ex:radial-mercer}
Let $(\mathcal S,d)$ be a metric space and let $\kappa:[0,\infty)\to\mathbb R$ be nonincreasing and
such that $G(x,y)=\kappa(d(x,y))$ is \emph{positive semidefinite}, meaning that for every finitely supported
function $q:\mathcal S\to\mathbb R$, $\sum_{x,y\in\mathcal S} q(x)\,G(x,y)\,q(y)\;\ge\;0.$
Then by the Moore--Aronszajn theorem \cite{Aronszajn1950} there exists a reproducing kernel Hilbert space $\mathcal H$
with feature map $\Phi(x):=G(\cdot,x)$ such that
\[
G(x,y)=\langle \Phi(x),\Phi(y)\rangle_{\mathcal H}, \qquad x,y\in\mathcal S.
\]
In particular,
\[
\|\Phi(x)\|_{\mathcal H}^2 = G(x,x) = \kappa(0)
\quad\text{for all }x\in\mathcal S,
\]
and therefore the gelation time satisfies
\[
(\kappa(0)M_2(0))^{-1}
\le
t_{\mathrm{gel}}
\le
(\|\mu_0\|_{\mathcal H}^2\,M_2(0))^{-1}.
\]
If the initial second moment distribution is a Dirac mass $p_0=\delta_x$, then
$\mu_0=\Phi(x)$ and hence $\|\mu_0\|_{\mathcal H}^2=\kappa(0)$, so the upper and lower
gelation bounds coincide. The converse is true when $\kappa$ is strictly decreasing.
\end{example}

\section{Proofs}\label{sec:proofs}
\begin{proof}[Proof of Proposition~\ref{prop:moment-equations}]
Fix $x\in\mathcal S$, and consider
$
m_1(x,t)=\sum_{i\ge1} i\lambda_t(i,x)$, as defined in \eqref{moments}.
We first show that $m_1(x,t)$ is well defined and constant on $[0,T_1)$.
Multiplying Eq.~\eqref{SE_prod} by $i$ and taking absolute values, we obtain
\begin{align*}
\sum_{i\ge1} i\bigl|\frac{d}{dt} \lambda_t(i,x)\bigr|
&\le
\sum_{i\ge1}\sum_{y\in\mathcal S}\sum_{j=1}^{i-1}
j^2(i-j)G(x,y)\lambda_t(j,x)\lambda_t(i-j,y)
\\
&\qquad
+\sum_{i\ge1}\sum_{y\in\mathcal S}\sum_{j\ge1}
i^2j G(x,y)\lambda_t(i,x)\lambda_t(j,y).
\end{align*}
In the first term, set $z=i-j$. Since all summands are nonnegative,
Tonelli's theorem yields
\begin{align*}
\sum_{i\ge1}\sum_{y\in\mathcal S}\sum_{j=1}^{i-1}
j^2(i-j)G(x,y)\lambda_t(j,x)\lambda_t(i-j,y)
&=
\sum_{y\in\mathcal S} G(x,y)
\sum_{j\ge1}\sum_{z\ge1}
j^2 z\lambda_t(j,x)\lambda_t(z,y)
\\
&=
\sum_{y\in\mathcal S} G(x,y)
\Bigl(\sum_{j\ge1} j^2\lambda_t(j,x)\Bigr)
\Bigl(\sum_{z\ge1} z\lambda_t(z,y)\Bigr).
\end{align*}
The second term is exactly the same quantity:
$$
\sum_{i\ge1}\sum_{y\in\mathcal S}\sum_{j\ge1}
i^2jG(x,y)\lambda_t(i,x)\lambda_t(j,y)
=
\sum_{y\in\mathcal S} G(x,y)
\Bigl(\sum_{i\ge1} i^2\lambda_t(i,x)\Bigr)
\Bigl(\sum_{j\ge1} j\lambda_t(j,y)\Bigr).
$$
Therefore
\begin{equation}
\label{eq:abs-der-bound}
\sum_{i\ge1} i\bigl|\frac{d}{dt} \lambda_t(i,x)\bigr|
\le
2\sum_{y\in\mathcal S} G(x,y)
\Bigl(\sum_{i\ge1} i^2\lambda_t(i,x)\Bigr)
\Bigl(\sum_{j\ge1} j\lambda_t(j,y)\Bigr),
\end{equation}
and the right-hand side is finite for $t\in[0,T_1)$ by \hyperlink{ass:A1}{(A1)}. Hence the series
$\sum_{i\ge1} i\frac{d}{dt}\lambda_t(i,x)$ converges absolutely, and we may differentiate term by term:
$$
\frac{d}{dt}m_1(x,t)=\sum_{i\ge1} i\frac{d}{dt}\lambda_t(i,x).
$$
Multiplying \eqref{SE_prod} by $i$ and summing over $i\ge1$, we get
\begin{align}
\frac{d}{dt}m_1(x,t)
&=
\sum_{i\ge1}\sum_{y\in\mathcal S}\sum_{j=1}^{i-1}
j^2(i-j)G(x,y)\lambda_t(j,x)\lambda_t(i-j,y)
\notag\\
&\qquad
-
\sum_{i\ge1}\sum_{y\in\mathcal S}\sum_{j\ge1}
i^2jG(x,y)\lambda_t(i,x)\lambda_t(j,y).
\label{eq:M1-derivative}
\end{align}
As above, after the change of variables $z=i-j$ in the first term,
Tonelli's theorem gives
\begin{align*}
\sum_{i\ge1}\sum_{y\in\mathcal S}\sum_{j=1}^{i-1}
j^2(i-j)G(x,y)\lambda_t(j,x)\lambda_t(i-j,y)
&=
\sum_{y\in\mathcal S} G(x,y)
\Bigl(\sum_{j\ge1} j^2\lambda_t(j,x)\Bigr)
\Bigl(\sum_{z\ge1} z\lambda_t(z,y)\Bigr),
\end{align*}
which coincides with the second term in
\eqref{eq:M1-derivative}. Therefore
$$
\frac{d}{dt}m_1(x,t)=0
\qquad\text{for all } t\in[0,T_1).
$$
Since the initial condition is monodisperse,
\begin{equation}\label{consm1}
m_1(0,x)=\sum_{i\ge1} i\lambda_0(i,x)=\lambda_0(1,x),
\end{equation}
which gives \eqref{eq:consm1}.

We now turn to $m_2(x,t)$. Assumption~\hyperlink{ass:A2}{(A2)} states that
\[
\sum_{y\in\mathcal S} G(x,y)m_2(x,t)m_2(y,t)<\infty
\qquad\text{for all } t\in[0,T_1).
\]
Together with the regularity of solutions to \eqref{SE_prod}
(see, e.g., \cite[Theorem~2.2]{norris-cluster-coag}), this implies that, for each fixed
$x\in\mathcal S$, the series
\[
\sum_{i\ge1} i^2\frac{d}{dt}\lambda_t(i,x)
\]
is absolutely convergent on $[0,T_1)$.
Hence we may again differentiate term-by-term and rearrange the resulting
sums.
 Using
\eqref{SE_prod}, we obtain
\begin{align*}
\frac{d}{dt}m_2(x,t)
&=
\sum_{i\ge1} i^2\frac{d}{dt}\lambda_t(i,x)\\
&=
\sum_{y\in\mathcal S}\sum_{j,k\ge1}
(j+k)^2\frac{j}{j+k}jkG(x,y)\lambda_t(j,x)\lambda_t(k,y)
-
\sum_{y\in\mathcal S}\sum_{j,k\ge1}
j^3kG(x,y)\lambda_t(j,x)\lambda_t(k,y).
\end{align*}
Since $(j+k)^2\frac{j}{j+k}jk=(j+k)j^2k=j^3k+j^2k^2,$ it follows that
\begin{align*}
\frac{d}{dt}m_2(x,t)
&=
\sum_{y\in\mathcal S}\sum_{j,k\ge1}
\bigl(j^3k+j^2k^2\bigr)G(x,y)\lambda_t(j,x)\lambda_t(k,y)
-
\sum_{y\in\mathcal S}\sum_{j,k\ge1}
j^3kG(x,y)\lambda_t(j,x)\lambda_t(k,y)\\
&=
\sum_{y\in\mathcal S}\sum_{j,k\ge1}
j^2k^2G(x,y)\lambda_t(j,x)\lambda_t(k,y)\\
&=
m_2(x,t)\sum_{y\in\mathcal S} G(x,y)m_2(y,t),
\end{align*}
which proves \eqref{eq:m2-local-prop}. Summing over $x\in\mathcal S$ and using Tonelli's theorem,
since all summands are nonnegative, yields \eqref{eq:M2-global}.
\end{proof}

\begin{proof}[Proof of Theorem~\ref{thm:spatial-multiplicative}]
Fix $k\ge1$. Plugging
\eqref{consm1} in the loss term of
\eqref{SE_prod}, we obtain
\begin{align*}
\sum_{y\in\mathcal S}\sum_{j\ge1}
kjG(x,y)\lambda_t(j,y)
=
k\sum_{y\in\mathcal S}G(x,y)\sum_{j\ge1}j\lambda_t(j,y)
=
k\sum_{y\in\mathcal S}G(x,y)\lambda_0(1,y)
=
kp(x),
\end{align*}
where $p(x):=\sum_{y\in\mathcal S}G(x,y)\lambda_0(1,y)$. Likewise, the gain term becomes
\begin{align*}
g_{k}(x,t):=\frac1k\sum_{y\in\mathcal S}G(x,y)\sum_{j=1}^{k-1}
j^2(k-j)\lambda_t(j,x)\lambda_t(k-j,y).
\end{align*}
This proves \eqref{eq:triangular-thm}--\eqref{eq:gk-thm}.
Now, $\lambda_t(1,x)$ can be explicitly computed.
For $k=1$, since the gain term is empty,
$$
\frac{d}{dt}\lambda_t(1,x)+p(x)\lambda_t(1,x)=0.
$$
Therefore
$$
\lambda_t(1,x)=\lambda_0(1,x)e^{-tp(x)}.
$$
From here, we can get a recursive formula for $k\ge2$.
Indeed, Eq.~\eqref{eq:gk-thm} that we have obtained so far, is a linear inhomogeneous ODE:
$$
\frac{d}{dt} \lambda_t(k,x)+kp(x)\lambda_t(k,x)=g_k(x,t).
$$
Since $\lambda_0(k,x)=0$ for all $k\ge2$, we obtain
$$
\lambda_t(k,x)
=
e^{-k t p(x)}
\int_0^t e^{k s p(x)}g_k(s,x)ds,
$$
which is \eqref{lambdak}.
\end{proof}

\begin{proof}[Proof of Theorem ~\ref{thm:riccati-general}]
Let $t$ be such that $M_2(t)\in(0,\infty)$, and recall
\begin{equation}\label{pt_recall}
p_t(x)=\frac{m_2(x,t)}{M_2(t)},
\qquad x\in\mathcal S.
\end{equation}
Since $m_2(x,t)\ge 0$ for every $x\in\mathcal S$, and
$$
\sum_{x\in\mathcal S} m_2(x,t)=M_2(t),
$$
it follows that $p_t$ is a probability measure on $\mathcal S$, and therefore
$
p_t(x)\ge 0$,
with 
$\sum_{x\in\mathcal S} p_t(x)=1.
$
By Proposition~\ref{prop:moment-equations}, we have
$$
\frac{d}{dt}M_2(t)
=
\sum_{x,y\in\mathcal S} G(x,y)m_2(x,t)m_2(y,t).
$$
Using the identity (from \eqref{pt_recall})
$
m_2(x,t)=M_2(t)p_t(x),
$
we obtain
\begin{align*}
\frac{d}{dt}M_2(t)
&=
\sum_{x,y\in\mathcal S} G(x,y)M_2(t)p_t(x)M_2(t)p_t(y)=
M_2(t)^2\sum_{x,y\in\mathcal S} G(x,y)p_t(x)p_t(y).
\end{align*}
Therefore
$$
\frac{d}{dt}M_2(t)=\alpha(t)M_2(t)^2,
$$
where
$$
\alpha(t):=\sum_{x,y\in\mathcal S} G(x,y)p_t(x)p_t(y).
$$
This proves \eqref{eq:Riccati-general}.
Since $M_2(t)>0$, we may rewrite \eqref{eq:Riccati-general} as
$
\frac{d}{dt}\Bigl(\frac{1}{M_2(t)}\Bigr)
=
-\alpha(t).
$
Integrating from $0$ to $t$, we get
$
\frac{1}{M_2(t)}=\frac{1}{M_2(0)}-\int_0^t \alpha(s)ds,
$
that is,
$$
M_2(t)
=
\frac{1}{M_2(0)^{-1}-\int_0^t\alpha(s)ds}.
$$
Finally, the above identity shows that $M_2(t)$ blows up when the denominator vanishes, hence
$$
t_{\mathrm{gel}}
=
\inf\Bigl\{
t>0:\int_0^t\alpha(s)ds=M_2(0)^{-1}
\Bigr\}.
$$
This concludes the proof.
\end{proof}

\begin{proof}[Proof of Corollary~\ref{cor:sharp-bounds}]
By Theorem~\ref{thm:riccati-general} (see Eq.~\eqref{eq:alfa}),
$
\alpha(t)=\sum_{x,y\in\mathcal S} G(x,y)p_t(x)p_t(y),
$
where $p_t$ is a probability distribution on $\mathcal S$. Hence
$p_t(x)p_t(y)\ge 0$ for all $x,y\in\mathcal S$, and
$$
\sum_{x,y\in\mathcal S} p_t(x)p_t(y)
=
\Bigl(\sum_{x\in\mathcal S} p_t(x)\Bigr)
\Bigl(\sum_{y\in\mathcal S} p_t(y)\Bigr)
=1.
$$
Therefore $\alpha(t)$ is a convex combination of the values
$\{G(x,y):x,y\in\mathcal S\}$, so that
$$
\inf_{x,y\in\mathcal S} G(x,y)
\le
\alpha(t)
\le
\sup_{x,y\in\mathcal S} G(x,y).
$$
Set
$
a:=\inf_{x,y\in\mathcal S} G(x,y)$,
$
b:=\sup_{x,y\in\mathcal S} G(x,y).
$
Then Theorem~\ref{thm:riccati-general} gives
$$
\frac{d}{dt}M_2(t)=\alpha(t)M_2(t)^2,
$$
hence
\begin{equation}\label{ineq:dM2}
aM_2(t)^2
\le
\frac{d}{dt}M_2(t)
\le
bM_2(t)^2.
\end{equation}
We first prove the lower bound on $t_{\mathrm{gel}}$. Since
$$
\frac{d}{dt}\Bigl(\frac{1}{M_2(t)}\Bigr)
=
-\frac{1}{M_2(t)^2}\frac{d}{dt}M_2(t),
$$
the inequalities in \eqref{ineq:dM2} imply
$
-b
\le
\frac{d}{dt}\Bigl(\frac{1}{M_2(t)}\Bigr)
\le
-a.
$
Integrating from $0$ to $t<t_{\mathrm{gel}}$, we obtain
$$
\frac{1}{M_2(0)}-bt
\le
\frac{1}{M_2(t)}
\le
\frac{1}{M_2(0)}-at.
$$
Now, for every $t<t_{\mathrm{gel}}$, one has $M_2(t)<\infty$, hence
$\frac1{M_2(t)}>0$. Therefore the quantity
$\frac{1}{M_2(0)}-bt$ must remain strictly below the first possible time at which
$\frac1{M_2(t)}$ can vanish. It follows that
$$
t_{\mathrm{gel}}
\ge
\frac{1}{bM_2(0)}
=
\frac{1}{\sup_{x,y\in\mathcal S} G(x,y)M_2(0)}.
$$
We next prove the upper bound. By the characterisation of the gelation time in
Theorem~\ref{thm:riccati-general},
$$
t_{\mathrm{gel}}
=
\inf
\Bigl\{
t>0:\int_0^t \alpha(s)ds=M_2(0)^{-1}
\Bigr\}.
$$
Since $\alpha(s)\ge a$ for all $s$, we have
$
\int_0^t \alpha(s)ds \geq at.
$
Hence, if $a>0$, then at time
$
t=\frac{1}{aM_2(0)}
$
the integral has already reached $M_2(0)^{-1}$, and therefore
$$
t_{\mathrm{gel}}
\le
\frac{1}{aM_2(0)}
=
\frac{1}{\inf_{x,y\in\mathcal S} G(x,y)M_2(0)}.
$$
If $a=0$, the right-hand side is understood as $+\infty$, so the estimate remains valid. This concludes the proof.
\end{proof}

\begin{proof}[Proof of Theorem \ref{thm:replicator-mercer}]
We start by differentiating with respect to $t$ equation \eqref{eq:pt} defining $p_t(x)$,
\[
\frac{d}{dt}p_t(x)
=
\frac{1}{M_2(t)}\frac{d}{dt}m_2(x,t)
-
p_t(x)\frac{1}{M_2(t)}\frac{d}{dt}M_2(t).
\]
Now expanding each term, by Proposition~\ref{prop:moment-equations},
\[
\frac{d}{dt}m_2(x,t)
=
m_2(x,t)\sum_{y\in\mathcal S}G(x,y)m_2(y,t)=
M_2^2(t)p_t(x)\sum_{y\in\mathcal S } G(x,y)p_t(y),
\]
and by Theorem~\ref{thm:riccati-general},
\[
\frac{1}{M_2(t)}\frac{d}{dt}M_2(t)
=
\alpha(t)M_2(t).
\]
Therefore,
\[
\frac{d}{dt}p_t(x)
=
M_2(t)p_t(x)
\left((Gp_t)(x)-\alpha(t)\right).
\]
To remove the factor $M_2(t)$, we introduce the rescaled time $\tau(t)=\int_0^t M_2(s){\rm d}s$. Since $M_2(t)\ge 0$ and is finite on
$[0,T_1)$, $\tau$ is well-defined and nondecreasing, which yields the claimed equation 
\[
\frac{d}{d\tau}p_\tau(x)=
\frac{1}{M_2(t)}
\frac{d}{dt}p_t(x)=p_\tau(x)\left((Gp_\tau)(x)-\alpha(\tau)\right).
\]
This turns out to be a replicator equation for evolving measure $p_\tau(x)$ with nonlinear payoff function $\alpha(\tau)$. We will now show that the latter is nondecreasing and use this property to obtain lower and upper bounds.
Since $G$ is symmetric by Mercer property and the series converges absolutely,
\[
\frac{d}{d\tau}\alpha(\tau)
=
2\sum_{x\in\mathcal S} \frac{d}{d\tau}p_\tau(x)\,(Gp_\tau)(x).
\]
Substituting \eqref{eq:replicator-mercer} and then using that $\sum_{x\in\mathcal S} p_\tau(x)(Gp_\tau)(x)=\alpha(\tau)$, we get
$$
\begin{aligned}
\frac{d}{d\tau}\alpha(\tau) 
&= 2\sum_{x\in\mathcal S} p_\tau(x)\bigl((Gp_\tau)(x)-\alpha(\tau)\bigr)(Gp_\tau)(x)\\
&=2\Big(\sum_{x\in\mathcal S} p_\tau(x)(Gp_\tau)(x)^2-\big(\sum_x p_\tau(x)(Gp_\tau)(x)\big)^2\Big)\\
&=2\,\mathrm{Var}_{p_\tau}(Gp_\tau)\ge 0.
\end{aligned}
$$
Moreover, since $d\tau/dt=M_2(t)\ge 0$, $\alpha$ is nondecreasing in $t$ as well, and we bound $\alpha(t)\ge\alpha(0)$.
Furthermore, we optimise \eqref{eq:alfa} using that $G$ is positive semidefinite; since $G$ is also bounded, say $|G_{x,y}|\le C$, we obtain
\[
\alpha(t)\le\alpha_{\max}
=
\sup_{\substack{q:\mathcal S\to[0,1]\\ \sum_x q(x)=1}}
\sum_{x,y\in\mathcal S} G_{x,y}q(x)q(y)
\le C\sum_{x,y\in\mathcal S} q(x)q(y)=C<\infty.
\]
To deduce the gelation bounds, rewrite the Riccati equation
\eqref{eq:Riccati-general} in the form
\[
\frac{d}{dt}\Bigl(\frac{1}{M_2(t)}\Bigr)=-\alpha(t).
\]
Since $\alpha(t)$ is nondecreasing and satisfies
$
\alpha(0)\le \alpha(t)\le \alpha_{\max}$.
Integrating over $[0,t]$ gives
\[
\frac{1}{M_2(0)}-\alpha_{\max} t \le \frac{1}{M_2(t)}\le\frac{1}{M_2(0)}-\alpha(0)t.
\]
The blow-up time of $M_2(t)$ is therefore bounded between the two corresponding linear hitting times, which yields \eqref{eq:bounds-mercer}. Finally note that if $\alpha_{\max}=0$, then $\alpha(\tau)\equiv 0$ and no gelation occurs.
\end{proof}

\section*{Acknowledgments}
 IK gratefully acknowledges support from Netherlands Research Organisation (NWO), research program VIDI, project number VI.Vidi.213.108.

\bibliographystyle{abbrv}
\bibliography{ref}

\end{document}